\documentclass[twoside,10pt]{article}
\usepackage{indentfirst}                              
\usepackage[colorlinks]{hyperref}
\hypersetup{linkcolor=black,filecolor=black,urlcolor=black, citecolor=black}  
\usepackage[numbers,sort&compress]{natbib} 
\usepackage[perpage,symbol*]{footmisc}  

\usepackage{amsmath} 
\usepackage{amsfonts}
\usepackage{amssymb} 
\usepackage{bm}           
\usepackage{mathrsfs} 
\usepackage{amsthm}  
\usepackage{accents}  

\usepackage[titletoc,title]{appendix}
\usepackage{titlesec}
\titleformat{\section}{\large\bfseries}{\thesection}{1em}{}
\titleformat{\subsection}[runin]{\bfseries}{\thesubsection.}{0.5em}{}[.]
\titleformat{\subsubsection}[runin]{\bfseries}{\thesubsubsection.}{0.4em}{}[.]

\DeclareMathOperator{\supp}{supp}

\def\d{\,\mathrm{d}}

\makeatletter
\def\wideubar{\underaccent{{\cc@style\underline{\mskip10mu}}}}
\def\Wideubar{\underaccent{{\cc@style\underline{\mskip8mu}}}}
\makeatother
\makeatletter
\def\widebar{\accentset{{\cc@style\underline{\mskip10mu}}}}
\def\Widebar{\accentset{{\cc@style\underline{\mskip8mu}}}}
\makeatother

\newcommand{\VERTiii}[1]{{\left\vert\kern-0.3ex\left\vert\kern-0.3ex\left\vert #1
 \right\vert\kern-0.3ex\right\vert\kern-0.3ex\right\vert}}
\newcommand{\VERT}{\vert\kern-0.3ex\vert\kern-0.3ex\vert}
\newcommand{\VERTl}{\left\vert\kern-0.3ex\left\vert\kern-0.3ex\left\vert}
\newcommand{\VERTr}{\right\vert\kern-0.3ex\right\vert\kern-0.3ex\right\vert}
\newcommand{\VERTbig}{\big\vert\kern-0.3ex\big\vert\kern-0.3ex\big\vert}
\newcommand{\VERTBig}{\Big\vert\kern-0.3ex\Big\vert\kern-0.3ex\Big\vert}

\makeatletter
\DeclareFontFamily{OMX}{MnSymbolE}{}
\DeclareSymbolFont{MnLargeSymbols}{OMX}{MnSymbolE}{m}{n}
\SetSymbolFont{MnLargeSymbols}{bold}{OMX}{MnSymbolE}{b}{n}
\DeclareFontShape{OMX}{MnSymbolE}{m}{n}{
 <-6>  MnSymbolE5 <6-7>  MnSymbolE6 <7-8>  MnSymbolE7 <8-9>  MnSymbolE8 <9-10> MnSymbolE9 <10-12> MnSymbolE10 <12->   MnSymbolE12
}{}
\DeclareFontShape{OMX}{MnSymbolE}{b}{n}{
 <-6>  MnSymbolE-Bold5 <6-7>  MnSymbolE-Bold6 <7-8>  MnSymbolE-Bold7 <8-9>  MnSymbolE-Bold8 <9-10> MnSymbolE-Bold9 <10-12> MnSymbolE-Bold10 <12->   MnSymbolE-Bold12
}{}

\let\llangle\@undefined
\let\rrangle\@undefined
\DeclareMathDelimiter{\llangle}{\mathopen}%
{MnLargeSymbols}{'164}{MnLargeSymbols}{'164}
\DeclareMathDelimiter{\rrangle}{\mathclose}%
{MnLargeSymbols}{'171}{MnLargeSymbols}{'171}
\makeatother

\numberwithin{equation}{section}

 \newtheorem{lemma}{Lemma}[section]
 
 \newtheorem{theorem}{Theorem}[section]
 \newtheorem{corollary}[lemma]{Corollary}

 \newtheorem{remark}{Remark}[section]

\newcommand{\R}{{\mathbb R}}

\def\ds{\displaystyle}

\def\duno{\partial_1}
\def\dquno{\partial^2_{11} }
\def\ddue{\partial_2}

\def\di{\partial_i}
\def\dj{\partial_j}

\def\dt{\partial_t}
\def\dqt{\partial^2_{tt}}

\def\F{\mathcal F}

\def\M{\mathsf{M}}
\def\Lc{\mathcal L}
\def\Rc{\mathcal R}
\def\V{\mathcal V}

\def\vv{\bf v}
\def\ww{\bf w}
\def\VV{\bf V}

\begin{document}
\title{{\bf Anisotropic regularity of linearized\\compressible vortex sheets}}
\author{
{\sc Paolo Secchi}\thanks{e-mail: paolo.secchi@unibs.it}\;,
{\footnotesize DICATAM, 
Universit\`a di Brescia, Via Valotti 9, 25133 Brescia, Italy}
}


\maketitle
\begin{abstract}
We are concerned with supersonic vortex sheets for the Euler equations of compressible inviscid fluids in two space dimensions. For the problem with constant coefficients, in \cite{MST18} the authors have derived a pseudo-differential equation which describes the time evolution of the discontinuity front of the vortex sheet. In agreement with the classical stability analysis, the problem is weakly stable if $|[v\cdot\tau]|>2\sqrt{2}\,c$, and
the well-posedness was obtained in standard weighted Sobolev spaces. 
 The aim of the present paper is to improve the result of \cite{MST18}, by showing the existence of the solution in function spaces with some additional weighted anisotropic regularity in the frequency space.

\vspace{2mm}

\vspace{2mm}

\end{abstract}

\tableofcontents

\section{Introduction}
\label{sect1}

We are concerned with the time evolution of vortex sheets for the Euler equations describing the motion of a compressible fluid.
Vortex sheets are interfaces between two incompressible or compressible flows across which there is a discontinuity in fluid velocity. Across a vortex sheet, the tangential velocity field has a jump, while the normal component of the flow velocity is continuous. The discontinuity in the tangential velocity field creates a concentration of vorticity along the interface. In particular, compressible vortex sheets are contact discontinuities to the Euler equations for compressible fluids and as such they are fundamental waves which play an important role in the study of general entropy solutions to multidimensional hyperbolic systems of conservation laws.

It was observed in \cite{M58MR0097930,FM63MR0154509}, by the normal mode analysis, that rectilinear vortex sheets for isentropic compressible fluids in two space dimensions are linearly stable when the Mach number $\M>\sqrt{2}$ and are violently unstable when $\mathsf{M}<\sqrt{2}$, while planar vortex sheets are always violently unstable in three space dimensions. This kind of instabilities is the analogue of the Kelvin--Helmholtz instability for incompressible fluids.
Artola and Majda \cite{AM87MR914450} studied certain instabilities of two-dimensional supersonic vortex sheets by analyzing the interaction with highly oscillatory waves through geometric optics. A rigorous mathematical theory on nonlinear stability and local-in-time existence of two-dimensional supersonic vortex sheets was first established by Coulombel--Secchi \cite{CS08MR2423311,CS09MR2505379} based on their linear stability results in \cite{CS04MR2095445} and a Nash--Moser iteration scheme.
We refer the reader to \cite{MST18} for more references.

The vortex sheet motion is a nonlinear hyperbolic problem with a characteristic free boundary.
The analysis of the linearized problem in \cite{CS04MR2095445} shows that the so-called Kreiss-Lopatinski\u{\i} condition holds in a weak sense, thus one can only obtain an \emph{a priori} energy estimate with a loss of derivatives with respect to the source terms. Because of this fact, the existence of the solution to the nonlinear problem is obtained in \cite{CS08MR2423311} by a Nash-Moser iteration scheme, with a loss of the regularity of the solution with respect to the initial data.

At the best of our knowledge the approach of \cite{CS04MR2095445,CS08MR2423311} is the only one known up to now, while it would be interesting to have different methods of proof capable to give the existence and possibly other properties of the solution.

In particular, the location of the discontinuity front of the vortex sheet is obtained through the jump conditions at the front, see \eqref{RH}, and is implicitly determined by the fluid motion in the interior regions, i.e. far from the front.
On the contrary, it would be interesting to find an \lq\lq{explicit}\rq\rq evolution equation for the vortex sheet, i.e. for the discontinuity front. In this regard we recall that in case of irrotational, incompressible vortex sheets, the location of the discontinuity front is described by the Birchhoff-Rott equation, see \cite{MR1688875,MB02MR1867882,MP94MR1245492}, whose solution is sufficient to give a complete description of the fluid motion through the Biot-Savart law.

In this paper we are concerned with supersonic vortex sheets for the Euler equations of compressible inviscid fluids in two space dimensions. For the problem with constant coefficients, in \cite{MST18} the authors derived an evolution pseudo-differential equation for the discontinuity front of the vortex sheet. In agreement with the classical stability analysis \cite{FM63MR0154509,M58MR0097930}, if the Mach number $\mathsf{M}<\sqrt{2}$ the symbol is elliptic and the problem is ill-posed. On the contrary, if $\mathsf{M}>\sqrt{2}$, then the problem is weakly stable, and the authors proved the well-posedness of the problem, by showing the existence of the solution in weighted Sobolev spaces.

The aim of the present paper is to improve the result of \cite{MST18} by showing the existence of the solution in a function space with some additional weighted anisotropic regularity in the frequency space. The main result of the paper is given by Theorem \ref{teoexist} in Section \ref{sezione3}. We hope that the introduction of such weighted anisotropic Sobolev will be an useful tool for the analysis of nonlinear compressible vortex sheets. This will be the object of a future investigation.

\subsection{The Eulerian description}

We consider the isentropic Euler equations in the whole plane $\R^2$. Denoting by ${\bf v}=(v_1,v_2) \in \R^2$ the
velocity of the fluid, and by $\rho$ its density, the equations read:
\begin{equation}
\label{euler}
\begin{cases}
\dt \rho +\nabla \cdot (\rho \, {\bf v}) =0 \, ,\\
\dt (\rho \, {\bf v}) +\nabla \cdot
(\rho \, {\bf v} \otimes {\bf v}) +\nabla \, p =0 \, ,
\end{cases}
\end{equation}
where $p=p(\rho)$ is the pressure law. In all this paper $p$ is a $C^\infty$ function of $\rho$,
defined on $]0,+\infty[$, and such that $p'(\rho)>0$ for all $\rho$. The speed of sound $c(\rho)$ in the
fluid is defined by the relation:
\begin{equation*}
\forall \, \rho>0 \, ,\quad c(\rho) :=\sqrt{p'(\rho)} \, .
\end{equation*}
It is a well-known fact that, for such a pressure law, \eqref{euler} is a strictly hyperbolic system in
the region $(t,x)\in\, ]0,+\infty[ \, \times \R^2$, and \eqref{euler} is also symmetrizable.

We are interested in solutions of \eqref{euler} that are smooth on either side of a smooth hypersurface  $\Gamma(t):=\{x=(x_1,x_2)\in \R^2 : F(t,x)=0\}=\{x_2=f(t,x_1)\}$ for each $t$ and that satisfy
suitable jump conditions at each point of the front $\Gamma (t)$.

Let us denote $\Omega^\pm(t):=\{(x_1,x_2)\in \R^2 :x_2\gtrless f(t,x_1)\}$. Given any
function $g$ we denote $g^\pm=g$ in $\Omega^\pm(t)$ and $[g]=g^+_{|\Gamma}-g^-_{|\Gamma}$ the jump across
$\Gamma (t)$.

We look for smooth solutions $({\vv}^\pm,\rho^\pm)$ of \eqref{euler} in $\Omega^\pm(t)$ and such that, at each time $t$,
the tangential velocity is the only quantity that experiments a jump across the curve $\Gamma (t)$. (Tangential
should be understood as tangential with respect to $\Gamma (t)$). The pressure and the normal velocity should be
continuous across $\Gamma (t)$. For such solutions, the jump conditions across $\Gamma(t)$ read:
\begin{equation*}
\sigma ={\vv}^\pm\cdot n  \, ,\quad [p]=0 \quad {\rm on } \;\Gamma (t) \, .
\end{equation*}
Here $n=n(t)$ denotes the outward unit normal on $\partial\Omega^-(t)$ and $\sigma$ denotes the velocity of
propagation of the interface $\Gamma (t)$. With our parametrization of $\Gamma (t)$, an equivalent formulation
of these jump conditions is
\begin{equation}
\label{RH}
\dt f ={\bf v}^+\cdot N ={\bf v}^-\cdot N \, ,\quad p^+ =p^-  \quad {\rm on }\;\Gamma (t) \, ,
\end{equation}
where
\begin{equation}\label{defN}
N=(-\duno f, 1)
\end{equation}
and $p^\pm=p(\rho^\pm)$. Notice that the function $f$ describing the discontinuity front is part of
the unknown of the problem, i.e. this is a free boundary problem.

For smooth solutions system \eqref{euler} can be written in the equivalent form
\begin{equation}
\label{euler1}
\begin{cases}
\dt \rho +({\bf v}\cdot\nabla) \rho +\rho \, \nabla\cdot{\bf v} =0 \, ,\\
\rho \,(\dt  {\bf v} +({\bf v}\cdot\nabla)
 {\bf v} ) +\nabla \, p =0 \, .
\end{cases}
\end{equation}
Because $ p'(\rho)>0 $, the function $p= p(\rho ) $ can be inverted and we can write $ \rho=\rho(p) $. Given a positive constant $ \bar{\rho}>0 $, we introduce the quantity $ P(p)=\log(\rho(p)/\bar{\rho}) $ and consider $ P $ as a new unknown. In terms of $ (P,\vv) $, the system \eqref{euler1} equivalently reads
\begin{equation}
\label{euler2}
\begin{cases}
\dt P +{\bf v}\cdot\nabla P + \nabla\cdot{\bf v} =0 \, ,\\
\dt  {\bf v} +({\bf v}\cdot\nabla)
{\bf v}  +c^2\,\nabla \, P =0 \, ,
\end{cases}
\end{equation}
where now the speed of sound is considered as a function of $ P,$ that is $ c=c(P) $. Thus our problem reads
\begin{equation}
\label{euler3}
\begin{cases}
\dt P^\pm +{\bf v}^\pm\cdot\nabla P^\pm + \nabla\cdot{\bf v}^\pm =0 \, ,\\
\dt  {\bf v}^\pm +({\bf v}^\pm\cdot\nabla)
{\bf v}^\pm  +c^2_\pm\,\nabla \, P^\pm =0 \, , \qquad {\rm in }\;  \Omega^\pm(t),
\end{cases}
\end{equation}
where we have set $ c_\pm=c(P^\pm) $.
 The jump conditions \eqref{RH}
 take the new form
\begin{equation}
\label{RH2}
\dt f ={\bf v}^+\cdot N ={\bf v}^-\cdot N \, ,\quad P^+ =P^-  \quad {\rm on }\;\Gamma (t) \, ,
\end{equation}


\section{Preliminary results}

For the reader's convenience we recall the main steps of the derivation of the evolution equation  for the discontinuity front of the vortex sheet with constant coefficients, obtained in \cite{MST18}. 

Given functions $ {\vv}^\pm, P^\pm$, we set
\begin{equation}
\begin{array}{ll}\label{defZ}
Z^\pm:=\dt {\vv}^\pm
+( {\vv}^\pm \cdot \nabla) {\vv}^\pm .
\end{array}
\end{equation}
Next, we study the behavior of $Z^\pm$  at $\Gamma(t)$. As in \cite{SWZ} we define
\begin{equation}
\begin{array}{ll}\label{deftheta}
\theta(t,x_1):= {\vv}^\pm(t,x_1,f(t,x_1))\cdot N(t,x_1),
\end{array}
\end{equation}
for $N$ given in \eqref{defN}.

\begin{lemma}[\cite{MST18}]\label{lemmaN}
Let $ f, {\vv}^\pm,  \theta$ be such that
\begin{equation}
\begin{array}{ll}\label{dtfthetavN}
\dt f=\theta= {\vv}^\pm\cdot N  \,  \qquad {\rm on }\;  \Gamma(t),
\end{array}
\end{equation}
and let $Z^\pm$ be defined by \eqref{defZ}.
Then
\begin{equation}
\begin{array}{ll}\label{applN}
Z^+ \cdot N
\ds = \dt\theta + 2 v_1^+\duno\theta + (v_1^+)^2 \dquno f \,,\\
Z^- \cdot N
\ds = \dt\theta + 2 v_1^-\duno\theta + (v_1^-)^2 \dquno f
 \quad {\rm on }\;  \Gamma(t) .
\end{array}
\end{equation}
\end{lemma}

\subsection{A first equation for the front}

We take the scalar product of the equation for $\vv^\pm$ in \eqref{euler3}, evaluated at $\Gamma(t)$, with the vector $N$. We get
\begin{equation*}
\big\{ Z^\pm  + c^2_\pm  \nabla P^\pm\big\} \cdot N =0\, \quad {\rm on} \; \Gamma(t) \, ,
\end{equation*}
and applying Lemma \ref{lemmaN} we obtain
\begin{equation}
\begin{array}{ll}\label{puntoN}
\ds  \dt\theta + 2 v_1^\pm\duno\theta + (v_1^\pm)^2 \dquno f + c^2_\pm  \nabla P^\pm \cdot N =0 \,\quad {\rm on} \; \Gamma(t) \, .
\end{array}
\end{equation}
Now we apply an idea from \cite{SWZ}. We take the {\it sum} of the "+" and "-" equations in \eqref{puntoN} to obtain
\begin{multline}\label{puntoN2}
\ds  2\dt\theta + 2 (v_1^++v_1^-)\duno\theta + ( (v_1^+)^2+(v_1^-)^2) \dquno f
+ c^2 \nabla (P^+ +  P^-)  \cdot N =0 \,\quad {\rm on} \; \Gamma(t) \, ,
\end{multline}
where we have denoted the common value at the boundary $c=c_{\pm|\Gamma(t)}=c(P^\pm_{|\Gamma(t)})$.
Next, following again \cite{SWZ}, we introduce the quantities
\begin{equation}
\label{defwV}
{\ww}=(w_1,w_2):=({\vv}^++{\vv}^-)/2, \qquad {\VV}=(V_1,V_2):=({\vv}^+-{\vv}^-)/2.
\end{equation}
Sustituting \eqref{defwV} in \eqref{puntoN2} gives
\begin{equation}\label{puntoN3}
\ds  \dt\theta + 2 w_1\duno\theta + (w_1^2 + V_1^2 )\dquno f
+\frac{c^2}2  \nabla (P^+ +  P^-)  \cdot N =0 \,\qquad {\rm on} \;  \Gamma(t) \, .
\end{equation}
Finally we substitute the boundary condition $ \theta=\dt f $ in \eqref{puntoN3} and we obtain
\begin{equation}\label{puntoN4}
\ds  \dqt f + 2 w_1\duno\dt f + (w_1^2 + V_1^2 )\dquno f
+\frac{c^2}2  \nabla (P^+ +  P^-)  \cdot N =0 \,\qquad {\rm on} \;  \Gamma(t) \, .
\end{equation}
\eqref{puntoN4} is a second order equation for the front $f$. However, it is nonlinearly coupled at the highest order with the other unknowns $ ({\vv}^\pm,P^\pm) $ of the problem through the last term in the left side of \eqref{puntoN4}. In order to find an evolution equation for $f,$ it is important to isolate the dependence of $f$ on $P^\pm$ at the highest order, i.e. up to lower order terms in $ ({\vv}^\pm,P^\pm) $.

Notice that \eqref{puntoN4} can also be written in the form
\begin{equation}\label{puntoN5}
\ds  (\dt  +  w_1\duno)^2 f +  V_1^2 \dquno f
+\frac{c^2}2  \nabla (P^+ +  P^-)  \cdot N -(\dt w_1+w_1\duno w_1)\duno f=0 \,\qquad {\rm on} \; \Gamma(t) \, .
\end{equation}

\subsection{The wave problem for the pressure}

Applying the operator $ \dt+{\vv}\cdot\nabla $ to the first equation of \eqref{euler2} and $ \nabla\cdot $ to the second one gives
\begin{equation*}\label{}
\begin{cases}
(\dt  +{\bf v}\cdot\nabla)^2 P + (\dt  +{\bf v}\cdot\nabla)\nabla\cdot{\bf v} =0 \, ,\\
\nabla\cdot(\dt  +{\bf v}\cdot\nabla)
{\bf v}  +\nabla\cdot(c^2\,\nabla \, P) =0 \, .
\end{cases}
\end{equation*}
The difference of the two equations gives the wave-type equation\footnote[1]{$[A,B]$ denotes the commutator $AB-BA$. In the paper we adopt the Einstein convention over repeated indices.}
\begin{equation}\label{wave0}
(\dt  +{\bf v}\cdot\nabla)^2 P - \nabla\cdot(c^2\,\nabla \, P) = -[\dt  +{\bf v}\cdot\nabla, \nabla\cdot\,]{\vv}=\di v_j\dj v_i.
\end{equation}
We repeat the same calculation for both $ ({\vv}^\pm,P^\pm) $.
As for the behavior at the boundary, we already know that
\begin{equation}\label{bc1}
 [P]=0 \, , \qquad
{\rm on }\;  \Gamma(t) \, .
\end{equation}
As a second boundary condition it is natural to add a condition involving the normal derivatives of $ P^\pm. $ We proceed as follows: instead of the {\it sum} of the equations \eqref{puntoN} as for \eqref{puntoN2}, we take the {\it difference} of the "+" and "-" equations in \eqref{puntoN} to obtain the jump of the normal derivatives $ \nabla P^\pm  \cdot N $,
\begin{equation}
\label{jumpQ}
[c^2 \nabla P \cdot N] =-[2 v_1\duno\theta + v_1^2 \dquno f] \qquad {\rm on} \; \Gamma(t) \, .
\end{equation}
Recalling that $ \theta=\dt f $, we compute
\begin{equation}\label{jumpQ1}
[2 v_1\duno\theta + v_1^2 \dquno f] = 4 V_1(\dt+w_1\duno)\duno f .
\end{equation}
Thus, from \eqref{jumpQ}, \eqref{jumpQ1} we get
\begin{equation}\label{bc2}
[c^2 \nabla P \cdot N] =-4 V_1(\dt+w_1\duno)\duno f \qquad {\rm on} \;  \Gamma(t) \, .
\end{equation}
Collecting \eqref{wave0} for $P^\pm$, \eqref{bc1}, \eqref{bc2} gives the coupled problem for the pressure
\begin{equation}\label{wave}
\begin{cases}
(\dt  +{\vv }^\pm\cdot\nabla)^2 P^\pm - \nabla\cdot(c^2_\pm\,\nabla \, P^\pm) =\F^\pm  &
{\rm in }\;  \Omega^\pm(t) \, ,\\
 [P]=0 \, ,\\
[c^2 \nabla P \cdot N] =-4 V_1(\dt+w_1\duno)\duno f & {\rm on} \; \Gamma(t) \, ,
\end{cases}
\end{equation}
where
\begin{equation*}\label{key}
\F^\pm:=\di v_j^\pm \dj v_i^\pm. 
\end{equation*}
Notice that $ \F $ can be considered a lower order term in the second order differential equation for $ P^\pm $, differently from the right-hand side of the boundary condition for the jump of the normal derivatives, which is of order two in $ f. $

\section{The coupled problem \eqref{puntoN5}, \eqref{wave} with constant coefficients. The main result}\label{sezione3}

We consider a problem obtained by linearization of equation \eqref{puntoN5} and system \eqref{wave} about the constant velocity ${\vv }^\pm=(\pm v,0)$, $v>0$, so that $w_1=0, V_1=v$, constant pressure $P^+=P^-$, and flat front $\Gamma=\{x_2=0\}$, so that $N=(0,1)$, that is we study the equations
\begin{equation}\label{puntoN6}
\ds  (\dqt    +  v^2 \dquno) f
+\frac{c^2}2 \ddue (P^+ +  P^-)   =0 \,\qquad {\rm if} \; x_2=0 \, ,
\end{equation}
\begin{equation}\label{wave2}
\begin{cases}
(\dt  \pm v \duno)^2 P^\pm - c^2\Delta \, P^\pm =\F^\pm  \quad &
{\rm if }\; x_2\gtrless0 \, ,\\
[P]=0 \, ,\\
[c^2\ddue P ] =-4 v\dt\duno f & {\rm if} \; x_2=0 \, .
\end{cases}
\end{equation}
In \eqref{puntoN6}, \eqref{wave2}, $c>0$ is a constant and $\F^\pm$ is a given source term. \eqref{puntoN6}, \eqref{wave2} form a coupled system for $f$ and $P^\pm$, obtained by retaining the highest order terms of \eqref{puntoN5} and \eqref{wave}.  We are interested to derive from \eqref{puntoN6}, \eqref{wave2} an evolution equation for the front $f$.

Given $\gamma\ge1$, we introduce $ \widetilde{f}:=e^{-\gamma t}f,\widetilde{P}^\pm:=e^{-\gamma t}P^\pm, \widetilde{\F}^\pm:=e^{-\gamma t}\F^\pm $ and consider the equations
\begin{equation}\label{puntoN7}
\ds  (\gamma+\dt )^2 \widetilde{f} +  v^2 \dquno \widetilde{f}
+\frac{c^2}2 \ddue (\widetilde{P}^+ +  \widetilde{P}^-)   =0 \,\qquad {\rm if} \; x_2=0 \, ,
\end{equation}
\begin{equation}\label{wave3}
\begin{cases}
(\gamma+\dt  \pm v\duno)^2 \widetilde{P}^\pm - c^2\Delta \, \widetilde{P}^\pm =\widetilde{\F}^\pm  \quad &
{\rm if }\; x_2\gtrless0 \, ,\\
[\widetilde{P}]=0 \, ,\\
[c^2\ddue \widetilde{P} ] =-4 v(\gamma+\dt)\duno \widetilde{f} & {\rm if} \; x_2=0 \, .
\end{cases}
\end{equation}
System \eqref{puntoN7}, \eqref{wave3} is equivalent to \eqref{puntoN6}, \eqref{wave2}. Let us denote by $ \widehat{f},\widehat{P}^\pm,\widehat{\F}^\pm $ the Fourier transforms of $ \widetilde{f},\widetilde{P}^\pm, \widetilde{\F}^\pm $ in $(t,x_1)$, with dual variables denoted by $(\delta,\eta)$, and set $\tau=\gamma+i\delta$. 
From \cite{MST18} we have the following result:
\begin{theorem}[\cite{MST18}]\label{teo_equ}
Let $\widetilde{\F}^\pm$ be such that
\begin{equation}\label{cond_infF}
\lim\limits_{x_2\to+\infty}\widehat{\F}^\pm(\cdot,\pm x_2)= 0 \, .
\end{equation}
Assume that $\widetilde{f},\widetilde{P}^\pm$ is a solution of \eqref{puntoN7}, \eqref{wave3} with
\begin{equation}\label{cond_inf}
\lim\limits_{x_2\to+\infty}\widehat{P}^\pm(\cdot,\pm x_2)= 0 \, .
\end{equation}
Then $f$ solves the second order pseudo-differential equation
\begin{equation}\label{equ_f}
\ds \left( \tau^2  +  v^2 \eta^2
\left(8\left( \frac{\tau/c}{\mu^+(\tau,\eta)+\mu^-(\tau,\eta)}\right)^2  -1 \right)\right) \widehat{f} + \frac{\mu^+\mu^-}{\mu^++\mu^-}\,M =0   \, ,
\end{equation}
where $\mu^\pm:=\sqrt{\left(\frac{\tau\pm iv\eta}{c}\right)^2+\eta^2}$ is
such that
$
\Re\mu^\pm>0$ if $\Re\tau>0$,
and
\begin{equation}\label{def_M}
M=M(\tau,\eta):= \frac{1}{\mu^+}\int_{0}^{+\infty}e^{-\mu^+ y}\widehat{\F}^+ (\cdot, y)\, dy -
\frac{1}{\mu^-}\int_{0}^{+\infty}e^{-\mu^- y}\widehat{\F}^- (\cdot,- y)\, dy \, .
\end{equation}
\end{theorem}
\begin{proof}
Equation \eqref{equ_f} follows from the explicit resolution of system \eqref{wave3} by Laplace-Fourier transform in $x_2$ and substitution in \eqref{puntoN7}. See \cite{MST18} for details.
\end{proof}
From the definition we see that the roots $ \mu^\pm $ are homogeneous functions of degree 1 in $(\tau, \eta)$. Therefore, the ratio $ (\tau/c)^2/(\mu^++\mu^-)^2 $ is homogeneous of degree 0. It follows that the symbol of \eqref{equ_f} is a homogeneous function of degree 2. In this sense \eqref{equ_f} represents a second order pseudo-differential equation for $f$.

In \cite{MST18} the authors prove the existence of the solution to equation \eqref{equ_f} in suitable weighted Sobolev spaces. The aim of the present paper is to improve the result of \cite{MST18} by showing the existence of the solution in a function space with some additional weighted anisotropic regularity in the frequency space. 

Our main result is given by the following result.

\begin{theorem}\label{teoexist}
	Assume $\mathsf{M}:=\frac{v}{c}>\sqrt{2}$, and let $ \F^+\in L^2(\R^+;H^s_\gamma(\R^2)) , \F^-\in L^2(\R^-;H^s_\gamma(\R^2))$. There exists a unique solution $f\in H^{s+1,\sigma}_\gamma(\R^2)$ of equation \eqref{equ_f}  satisfying the estimate
	\begin{equation}\label{stimafF1}
	\gamma \|f\|^2_{H^{s+1,\sigma}_\gamma(\R^2)} \le C\left( \|\F^+\|^2_{L^2(\R^+;H^s_\gamma(\R^2))}+\|\F^-\|^2_{L^2(\R^-;H^s_\gamma(\R^2))}\right) , \qquad\forall \gamma\ge1\, ,
	\end{equation}
	for a suitable constant $C>0$ independent of $\F^\pm$ and $\gamma$.
\end{theorem}
See Section \ref{sec2.w} for the definition of the function spaces.
We refer the reader to Remark \ref{ell_hyp} for a discussion about the different cases $ \mathsf{M}\gtrless\sqrt{2} $ in relation with the classical stability analysis \cite{CS04MR2095445,FM63MR0154509,M58MR0097930,S00MR1775057}. Corollary \ref{corolla2} shows that the existence result of \cite{MST18} is a consequence of the Theorem \ref{teoexist}.

\subsection{The symbol of the pseudo-differential equation \eqref{equ_f}}

 Let us denote the symbol of \eqref{equ_f} by $ \Sigma: $
\begin{equation}\label{def_Sigma}
\Sigma(\tau,\eta):= \tau^2  +  v^2 \eta^2
\left(8\left( \frac{\tau/c}{\mu^+(\tau,\eta)+\mu^-(\tau,\eta)}\right)^2  -1 \right).
\end{equation}
In order to take the homogeneity into account, we define the hemisphere:
\begin{align*}
\Xi_1:=\left\{(\tau,\eta)\in \mathbb{C}\times\mathbb{R}\, :\,
|\tau|^2+\eta^2=1,\Re \tau\geq 0 \right\},
\end{align*}
and the set of ``frequencies'':
\begin{align*}
\Xi:=\left\{(\tau,\eta)\in \mathbb{C}\times\mathbb{R}\, :\,
\Re \tau\geq 0, (\tau,\eta)\ne (0,0) \right\}=(0,\infty)\cdot\Xi_1  \,.
\end{align*}

The symbol $ \Sigma $ is an analytic function as long as $\Re\tau>0$ (so that $
\Re\mu^\pm>0$). Moreover, as shown in \cite{MST18}, in $\Xi_1$ it admits a continuous extension to all points with $\Re\tau=0$, specifically at points $(0,\eta)$ where $\mu^++\mu^-$ vanishes.

The aim of the present paper is to show the existence of the solution in a function space with additional weighted anisotropic regularity in the frequency space. For the definition of such function space we first recall a result from \cite{MST18} about the roots of the symbol $ \Sigma $.

\begin{lemma}[\cite{MST18}]\label{zeri_Sigma}
Let $ \Sigma(\tau,\eta) $ be the symbol defined in \eqref{def_Sigma}, for $ (\tau,\eta)\in\Xi. $
\begin{itemize}
\item[(i)] If $\M<\sqrt{2}$, then $ \Sigma(\tau,\eta)=0 $ if and only if
\[
 \tau=cY_1|\eta|  \qquad \forall\eta\not=0\, ,  \]
where
\[ Y_1= \sqrt{-\left(\M^2+1\right) + \sqrt{4\M^2+1}}\,  .  \]

\item[(ii)] If $\M>\sqrt{2}$, then $ \Sigma(\tau,\eta)=0 $ if and only if
\[ \tau=\pm icY_2\eta \qquad \forall\eta\not=0 \, ,  \]
where
\begin{equation}\label{defY2}
 Y_2= \sqrt{\M^2+1 - \sqrt{4\M^2+1}}\,  .
\end{equation}
Each of these roots is simple. For instance, there exists a neighborhood $\V$ of $( icY_2\eta,\eta)$ in $\Xi_1$ and a $C^\infty$ function $H$ defined on $\V$ such that
\[ \Sigma(\tau,\eta)=(\tau-icY_2\eta)H(\tau,\eta), \quad H(\tau,\eta)\not=0 \quad\forall (\tau,\eta)\in\V.
 \]
 A similar result holds near $(-icY_2\eta,\eta)\in\Xi_1$.
\end{itemize}
\end{lemma}

\begin{remark}\label{ell_hyp}
(i) Recall that the equation \eqref{equ_f} was obtained by taking the Fourier transform with respect to $(t,x_1)$ of \eqref{puntoN7}, \eqref{wave3}, which corresponds to taking the Laplace transform with respect to $t$ and the Fourier transform with respect to $x_1$ of \eqref{puntoN6}, \eqref{wave2}. Taking the Fourier transform with respect to $t$ of \eqref{puntoN6}, \eqref{wave2} corresponds to the case $\gamma=\Re\tau=0 $, i.e. $ (\tau,\eta)=(i\delta,\eta) $.

If $\M<\sqrt{2}$, from Lemma \ref{zeri_Sigma} the symbol $ \Sigma(\tau,\eta) $ only vanishes in points $ (\tau,\eta)$ with $\tau\in\R, \tau>0$. It follows that $ \Sigma(i\delta,\eta)\not=0 $ for all $(\delta,\eta)\in\R^2$. Therefore the symbol is elliptic, according to the standard definition. In this case planar vortex sheets are violently unstable, see \cite{S00MR1775057}.

(ii) If $\M>\sqrt{2}$, $ \Sigma(\tau,\eta) $ vanishes in points $ (\tau,\eta)$ with $\Re\tau=0$, that is on the boundary of the frequency set $\Xi$. In this case planar vortex sheets are known to be weakly stable, in the sense that the so-called Lopatinski\u{\i} condition holds in a weak sense, see \cite{CS04MR2095445,FM63MR0154509,M58MR0097930,S00MR1775057}. For this case we expect a loss of derivatives for the solution with respect to the data.
\end{remark}

From now we denote by $C,C'$ generic positive constants (independent of $\gamma$) that may vary from line to line, or even in the same line.
Let us define the weight function
\begin{equation}\label{defsigma}
\sigma(\tau,\eta):=\frac{(\tau-icY_2\eta)(\tau+icY_2\eta)}{(|\tau|^2+\eta^2)^{1/2}} \qquad(\tau,\eta)\in\Xi\,,
\end{equation}
where $Y_2$ is defined in \eqref{defY2}. Notice that $\sigma(\tau,\eta)$ is a homogeneous function of degree 1 in $(\tau,\eta)\in\Xi$ and vanishes if and only if $(\tau,\eta)=(\pm icY_2\eta ,\eta),\,\eta\not=0$, that is at points where $\Sigma(\tau,\eta)=0$.
\begin{lemma}\label{propsigma}
Let $\M>\sqrt{2}$. Let $ \sigma(\tau,\eta) $ be defined by \eqref{defsigma}.
\begin{itemize}
\item[(i)] If $(\tau_0,\eta_0)=(+ icY_2\eta_0 ,\eta_0),\,\eta_0\not=0$, there exists a small neighborhood $ \V_+ $ of $(\tau_0,\eta_0)$ in $\Xi_1$ such that
\[ 
|\sigma(\tau,\eta)|\ge C|\tau-icY_2\eta|\geq C\gamma \qquad \forall(\tau,\eta)\in \V_+\cdot\R^+.
\]
\item[(ii)] If $(\tau_0,\eta_0)=(- icY_2\eta_0 ,\eta_0),\,\eta_0\not=0$, there exists a small neighborhood $ \V_- $ of $(\tau_0,\eta_0)$ in $\Xi_1$ such that
\[ 
|\sigma(\tau,\eta)|\ge C|\tau+icY_2\eta|\geq C\gamma \qquad \forall(\tau,\eta)\in \V_-\cdot\R^+.
\]
\item[(iii)] If $(\tau_0,\eta_0)\not=(\pm icY_2\eta_0 ,\eta_0),\,\eta_0\not=0$, there exists a small neighborhood $ \V $ of $(\tau_0,\eta_0)$ in $\Xi_1$ such that
\[ 
C_1\gamma\le C_1\Lambda(\tau,\eta)\le |\sigma(\tau,\eta)|\le  C_2\Lambda(\tau,\eta)\qquad \forall(\tau,\eta)\in \V\cdot\R^+.
\]
\end{itemize}
\end{lemma}
\begin{proof}
(i) In a small neighborhood $\V_+$ of $(\tau_0,\eta_0)$ in $\Xi_1$ we may assume that 
$ \frac{\tau+icY_2\eta}{(|\tau|^2+\eta^2)^{1/2}}\not=0.
$ Because it is a homogeneous function of degree 0 there exist constants $C,C'>0$ such that
\begin{equation}\label{ratioomog0}
C \le \left| \frac{\tau+icY_2\eta}{(|\tau|^2+\eta^2)^{1/2}}\right| \leq C'\qquad \forall(\tau,\eta)\in \V_+\cdot\R^+.
\end{equation}
\[  \]
Then 
\[ 
|\sigma(\tau,\eta)|\ge C|\tau-icY_2\eta|\geq C\gamma \qquad \forall(\tau,\eta)\in \V_+\cdot\R^+.
\]
(ii) The proof is similar to the previous case.

\noindent(iii)  In a small neighborhood of $(\tau_0,\eta_0)$ we have $\sigma(\tau,\eta)\not=0$. Then the result follows as before because $\sigma$ is homogeneous of degree 1.
\end{proof}

\subsection{Weighted Sobolev spaces and norms}\label{sec2.w}
We are going to introduce certain weighted Sobolev spaces in order to prove Theorem \ref{teoexist}.
For all $s\in\mathbb{R}$ and for all $\gamma\geq 1$,
the usual Sobolev space $H^s(\mathbb{R}^2)$ is equipped with the following norm:
\begin{align*}
\|v\|_{s,\gamma}^2:=\frac{1}{(2\pi)^2} \iint_{\mathbb{R}^2}\Lambda^{2s}(\tau,\eta) |\widehat{v}(\delta,\eta)|^2\d \delta \d\eta,\qquad
\Lambda^{s}(\tau,\eta):=(\gamma^2+\delta^2+\eta^2)^{\frac{s}{2}}=(|\tau|^2+\eta^2)^{\frac{s}{2}},
\end{align*}
where $\widehat{v}(\delta,\eta)$ is the Fourier transform of $v(t,x_1)$ and $ \tau=\gamma+i\delta $.
We will abbreviate the usual norm of $L^2(\mathbb{R}^2)$ as
\begin{align*}
\|\cdot\|:=\|\cdot\|_{0,\gamma}\, .
\end{align*}
The scalar product in $L^2(\mathbb{R}^2)$ is denoted as follows:
\begin{align*}
\langle a,b\rangle:=\iint_{\mathbb{R}^2} a(x)\overline{b(x)}\d x,
\end{align*}
where $\overline{b(x)}$ is the complex conjugation of $b(x)$.

For $s\in\mathbb{R}$ and $\gamma\geq 1$, we introduce the weighted Sobolev
space $H^{s}_{\gamma}(\mathbb{R}^2)$ as
\begin{align*}
H^{s}_{\gamma}(\mathbb{R}^2)&:=\left\{
u\in\mathcal{D}'(\mathbb{R}^2)\,:\, \mathrm{e}^{-\gamma t}u(t,x_1)\in
H^{s}(\mathbb{R}^2) \right\},
\end{align*}
and its norm $\|u\|_{H^{s}_{\gamma}(\mathbb{R}^2)}:=\|\mathrm{e}^{-\gamma t}u\|_{s,\gamma}=\frac{1}{2\pi}\|\Lambda^{s}\widehat{u}\|$.
We write $L^2_{\gamma}(\mathbb{R}^2):=H^0_{\gamma}(\mathbb{R}^2)$ and $\|u\|_{L^2_{\gamma}(\mathbb{R}^2)}:=\|\mathrm{e}^{-\gamma t}u\|$.

Moreover, we introduce the weighted anisotropic space $H^{s,\sigma}_{\gamma}(\mathbb{R}^2)$ as
\begin{align*}
H^{s,\sigma}_{\gamma}(\mathbb{R}^2)&:=\left\{
u\in\mathcal{D}'(\mathbb{R}^2)\,:\, \sigma\Lambda^s\widehat{u}\in
L^{2}(\mathbb{R}^2) \right\},
\end{align*}
and its norm 
\[
\|u\|^2_{H^{s,\sigma}_{\gamma}(\mathbb{R}^2)}:=\frac{1}{(2\pi)^2}\|\sigma\Lambda^s\widehat{u}\|^2 
=\frac{1}{(2\pi)^2} \iint_{\mathbb{R}^2}|\sigma(\tau,\eta)\Lambda^{s}(\tau,\eta) \widehat{u}(\delta,\eta)|^2\d \delta \d\eta \, .
\]
Using Lemma \ref{propsigma} it is easily seen that $C\gamma\le|\sigma(\tau,\eta)|\le C\Lambda(\tau,\eta)$ for all $(\tau,\eta)\in\Xi$, for a suitable constant $C>0$. It follows that $H^{s+1}_{\gamma}(\mathbb{R}^2)\subset H^{s,\sigma}_{\gamma}(\mathbb{R}^2)\subset H^{s}_{\gamma}(\mathbb{R}^2)$ and
\begin{equation}\label{stimaspazi}
\gamma\|u\|_{H^{s}_{\gamma}(\mathbb{R}^2)}\le C\|u\|_{H^{s,\sigma}_{\gamma}(\mathbb{R}^2)}\le C\|u\|_{H^{s+1}_{\gamma}(\mathbb{R}^2)}. 
\end{equation} 
Lemma \ref{propsigma} also shows that the space $H^{s,\sigma}_{\gamma}(\mathbb{R}^2)$ is anisotropic in the frequency space in that along the directions $(\tau,\eta)=(\pm icY_2\eta,\eta),\eta\not=0,$ functions from $H^{s,\sigma}_{\gamma}(\mathbb{R}^2)$ have a $H^{s}_{\gamma}$ regularity, while off such directions they are $H^{s+1}_{\gamma}$.

We define $L^2(\mathbb{R}^\pm;H^{s}_{\gamma}(\mathbb{R}^2))$
as the spaces of distributions with finite norm
\begin{align*}
\|u\|_{L^2(\mathbb{R}^\pm;H^s_{\gamma}(\mathbb{R}^2))}^2:=\int_{\mathbb{R}^+}\|u(\cdot,\pm x_2)\|_{H^s_{\gamma}(\mathbb{R}^2)}^2\d x_2
\, .
\end{align*}

\section{Proof of Theorem \ref{teoexist}}

\begin{theorem}\label{teostimaSs}
	Let $\M>\sqrt{2}$. Let $ \Sigma(\tau,\eta) $ be the symbol defined in \eqref{def_Sigma} and $ \sigma(\tau,\eta) $ be defined by \eqref{defsigma}. 
	There exist constants $C_1,C_2,C_3>0$ such that
	\begin{equation}\label{stimaSs}
	|\sigma| \Lambda \le C_1|\Sigma|\le C_2|\sigma| \Lambda\le C_3\Lambda^2 \qquad \forall (\tau,\eta)\in\Xi.
	\end{equation}
\end{theorem}
\begin{proof}
	The study of $\Sigma$ in the proof of Lemma \ref{zeri_Sigma} implies that for all $(\tau_0,\eta_0)\in\Xi_1$, there exists a neighborhood $\V$ of $(\tau_0,\eta_0)$ with suitable properties, as explained in the following. Because $\Xi_1$ is a $C^\infty$ compact manifold, there exists a finite covering $(\V_1,\dots,\V_I)$ of $\Xi_1$ by such neighborhoods, and a smooth partition of unity $(\chi_1,\dots,\chi_I)$ associated with this covering.
	The $\chi_i's$ are nonnegative $C^\infty$ functions with
	\[
	\supp\chi_i\subset\V_i, \qquad \sum_{i=1}^I\chi_i^2=1.
	\]
	We consider two different cases.
	
	{\it In the first case}
	$\V_i$ is a neighborhood of an {\it elliptic} point, that is a point $ (\tau_0,\eta_0) $ where $ \Sigma(\tau_0,\eta_0)\not=0. $ By taking $\V_i$ sufficiently small we may assume that $ \Sigma(\tau,\eta)\not=0 $ in the whole neighborhood $\V_i$, and there exist two positive constants $C,C'$ such that $$ C\le|\Sigma(\tau,\eta)|\le C' \qquad \forall (\tau,\eta)\in\V_i\,.
	$$
	$ \Sigma(\tau,\eta)$ is a homogeneous function of degree 2 on $\Xi$, so we have
	\begin{equation}\label{stima_ellip0}
	C\Lambda^2(\tau,\eta)\le|\Sigma(\tau,\eta)|\le C'\Lambda^2(\tau,\eta) \qquad \forall (\tau,\eta)\in\V_i\cdot\R^+\,.
	\end{equation}
	From Lemma \ref{propsigma} (iii) we may also assume that $ \sigma(\tau,\eta)\not=0 $ on $\V_i$. Since $ \sigma(\tau,\eta)$ is a continuous homogeneous function of degree 1 on $\Xi$ we get in a similar way
	\begin{equation}\label{stima_ellip00}
	C\Lambda(\tau,\eta)\le|\sigma(\tau,\eta)|\le C'\Lambda(\tau,\eta) \qquad \forall (\tau,\eta)\in\V_i\cdot\R^+\,.
	\end{equation}
	Let us extend the associated function $\chi_i$ to the whole set of frequencies $\Xi$, as a homogeneous mapping of degree 0 with respect to $(\tau,\eta)$.
	From \eqref{stima_ellip0}, \eqref{stima_ellip00}
	we deduce that
	\begin{equation}\label{stima_ellip}
	\chi_i|\sigma|\Lambda\le
	C\chi_i|\Sigma|\le C\chi_i |\sigma|\Lambda\le C\chi_i\Lambda^2 \qquad \forall (\tau,\eta)\in\V_i\cdot\R^+\,.
	\end{equation}
	
	{\it In the second case}\label{first}
	$\V_i$ is a neighborhood of a {\it root} of the symbol $ \Sigma $, i.e. a point $ (\tau_0,\eta_0) $ where $ \Sigma(\tau_0,\eta_0)=0. $ For instance we may assume that $ (\tau_0,\eta_0)=(icY_2\eta_0,\eta_0), \eta_0\not=0$, see Lemma \ref{zeri_Sigma}; a similar argument applies for the other family of roots $(\tau,\eta)=(-icY_2\eta,\eta)$.
	According to Lemma \ref{zeri_Sigma} we may assume that on $\V_i$ it holds
	\[ \Sigma(\tau,\eta)=(\tau-icY_2\eta)H(\tau,\eta), \quad H(\tau,\eta)\not=0 \quad\forall (\tau,\eta)\in\V_i.
	\]
	We extend the associated function $\chi_i$ to the whole set of frequencies $\Xi$, as a homogeneous mapping of degree 0 with respect to $(\tau,\eta)$.
	Because $ H(\tau,\eta)\not=0 $ on $\V_i$, there exists positive constants $C,C'$ such that $$ C\le |H(\tau,\eta)|\le C' \qquad \forall (\tau,\eta)\in\V_i\,.
	$$
	$H(\tau,\eta)$ is a homogeneous function of degree 1 on $\Xi$, so we have
	\[
C\Lambda(\tau,\eta)\le	|H(\tau,\eta)|\le C'\Lambda(\tau,\eta) \qquad \forall (\tau,\eta)\in\V_i\cdot\R^+\,.
	\]	
	We get
	\begin{equation}\label{stima_nellip001}
	|\Sigma(\tau,\eta)|=|(\tau-icY_2\eta)H(\tau,\eta)|\ge C|\tau-icY_2\eta|\Lambda(\tau,\eta) \qquad \forall (\tau,\eta)\in\V_i\cdot\R^+\,.
	\end{equation}
	From Lemma \ref{propsigma} (i), see also \eqref{ratioomog0}, and \eqref{stima_nellip001} we obtain
	\begin{equation}\label{stima_nellip0}
	|\Sigma(\tau,\eta)|\ge  C|\tau-icY_2\eta|
	\left| \frac{\tau+icY_2\eta}{(|\tau|^2+\eta^2)^{1/2}}\right|
	\Lambda(\tau,\eta)=C|\sigma(\tau,\eta)|\Lambda(\tau,\eta)
	\qquad \forall (\tau,\eta)\in\V_i\cdot\R^+\,.
	\end{equation}
	In a similar way we get on the other side
	\begin{equation}\label{stima_nellip00}
	|\Sigma(\tau,\eta)|\le  C|\tau-icY_2\eta|
	\Lambda(\tau,\eta)\le C|\sigma(\tau,\eta)|\Lambda(\tau,\eta)
	\qquad \forall (\tau,\eta)\in\V_i\cdot\R^+\,.
	\end{equation}
	Also in this case we deduce that
	\begin{equation}\label{stima_nellip}
	\chi_i|\sigma|\Lambda\le
C\chi_i|\Sigma|\le C\chi_i |\sigma|\Lambda\le C\chi_i\Lambda^2 \qquad \forall (\tau,\eta)\in\V_i\cdot\R^+\,.
	\end{equation}
	
	In conclusion, adding up the square of \eqref{stima_ellip} and \eqref{stima_nellip}, and using that the $ \chi_i $'s form a partition of unity gives \eqref{stimaSs}. 
\end{proof}

Let us denote by $S$ the operator with symbol $\Sigma$, that is defined by
\begin{equation}\label{def_S}
\widehat {Sf}(\tau,\eta)=\Sigma(\tau,\eta)\widehat f(\tau,\eta)  \qquad (\tau,\eta)\in\Xi\, .
\end{equation} 

Theorem \ref{teostimaSs} gives the following alternative characterization of the space $H^{s+1,\sigma}_\gamma(\R^2)$.
\begin{lemma}
	Let $\M>\sqrt{2}$. Let $ S $ be the operator defined in \eqref{def_S}, $ \sigma(\tau,\eta) $ be defined by \eqref{defsigma} and $s\in\R, \gamma\ge1$.
	Then
	\[
	H^{s+1,\sigma}_\gamma(\R^2)= \left\lbrace  f\in H^{s+1}_\gamma(\R^2)\, :\, S f \in H^{s}_\gamma(\R^2)\right\rbrace 
	.
	\]
\end{lemma}
\begin{proof}
Multiplying estimate \eqref{stimaSs} by $ \Lambda^s $ gives 
\begin{equation}\label{stimaSs2}
|\sigma| \Lambda^{s+1} \le C_1|\Lambda^s\Sigma|\le C_2|\sigma| \Lambda^{s+1} \qquad \forall (\tau,\eta)\in\Xi.
\end{equation}
Then from \eqref{stimaspazi}, \eqref{stimaSs2} we get
\[
\gamma\|f\|_{H^{s+1}_{\gamma}(\mathbb{R}^2)}
\le C\|f\|_{H^{s+1,\sigma}_{\gamma}(\mathbb{R}^2)}=C\|\sigma\Lambda^{s+1}\widehat{f}\|\le C\|\Lambda^{s}\Sigma\widehat{f} \|=C\|Sf\|_{H^{s}_{\gamma}(\mathbb{R}^2)}\, ,
\]
and
\[
 \|Sf\|_{H^{s}_{\gamma}(\mathbb{R}^2)}=\frac{1}{2\pi}\|\Lambda^{s}\Sigma\widehat{f}\ \|\le  C\|\sigma\Lambda^{s+1}\widehat{f}\|=C\|f\|_{H^{s+1,\sigma}_{\gamma}(\mathbb{R}^2)}
\, ,
\]
which yield the thesis.
\end{proof}

Consider the equation
	\begin{equation}\label{def_g}
\Sigma(\tau,\eta)\widehat f(\tau,\eta)=\widehat g(\tau,\eta) \qquad (\tau,\eta)\in\Xi\, ,
\end{equation}
where $ \widehat{g} $ is the Fourier transform of $\widetilde{g}:= e^{-\gamma t}g. $ 
\begin{corollary}\label{teofg}
	Assume $\M>\sqrt{2}$. Given any $f\in H^{s+2}_\gamma(\R^2)$, let $g\in H^s_\gamma(\R^2)$ be the function defined by \eqref{def_g}.
	Then there exists a positive constant $C$ such that for all $\gamma\ge1$ the following estimate holds
	\begin{equation}\label{stimafg}
 \|f\|_{H^{s+1,\sigma}_\gamma(\R^2)} \le C \|g\|_{H^s_\gamma(\R^2)} \, .
	\end{equation}
\end{corollary}
\begin{proof}
Multiplying \eqref{stimaSs} by $ \Lambda^{s}(\tau,\eta) $ and using \eqref{def_g} gives
\begin{equation}\label{stimapointwise}
C|\sigma(\tau,\eta)\Lambda^{s+1}(\tau,\eta)\widehat f(\tau,\eta)|^2\le
|\Lambda^{s}(\tau,\eta)\widehat g(\tau,\eta)|^2  \qquad \forall (\tau,\eta)\in\Xi\,.
\end{equation}
Then, integrating with respect to $ (\delta,\eta)\in\R^2 $ and using Plancherel's theorem yields the estimate
\eqref{stimafg} for a suitable constant $C$.
\end{proof}

In the following theorem we prove the existence of the solution $f$ to equation \eqref{def_g}.

\begin{theorem}\label{teoexistfg2}
	Assume $\M>\sqrt{2}$. Let $\Sigma$ be the symbol defined by \eqref{def_Sigma} and $s\in\R, \gamma\ge1$. Given any $g\in H^s_\gamma(\R^2)$ there exists a unique solution $f\in H^{s+1,\sigma}_\gamma(\R^2)$ of equation \eqref{def_g}, satisfying the estimate
\eqref{stimafg}.
\end{theorem}

\begin{proof}
We use a duality argument. Let us denote by $ \Sigma^* $ the symbol of the adjoint operator $ S^* $, such that
\begin{align*}
\langle \Sigma\widehat{f},\widehat{h}\rangle=\langle \widehat{f},\Sigma^*\widehat{h}\rangle
\end{align*}
for $ f,h $ sufficiently smooth. From the definition \eqref{def_Sigma} we easily deduce that
\begin{equation}\label{equSS*}
\Sigma^\ast(\tau,\eta)=\Sigma(\bar{\tau},\eta)\,.
\end{equation}
Thus, from Theorem \ref{teofg}, see in particular \eqref{stima_ellip0}, \eqref{stima_nellip0}, \eqref{stimapointwise} and notice that $\sigma(\overline{\tau}, \eta)=\overline{\sigma}({\tau}, \eta)$, $\Lambda(\overline{\tau}, \eta)=\Lambda({\tau}, \eta)$, we obtain the estimate
 \begin{equation*}
|\sigma(\tau,\eta)\Lambda(\tau,\eta)\widehat h(\tau,\eta)|^2\le
C|\Sigma^\ast(\tau,\eta)\widehat h(\tau,\eta)|^2  \,,
 \end{equation*}
 which gives by integration in $ (\delta,\eta) $
\begin{equation}\label{stimaSigma*}
\|\sigma\Lambda\widehat h \|\le
C\|\Sigma^\ast\widehat h\|  \,.
\end{equation}
We compute
\begin{align}\label{duality}
\left|\langle \widehat{g},\widehat{h}\rangle\right|=\left| \langle \Lambda^s\widehat{g},\Lambda^{-s}\widehat{h}\rangle\right|
\le\|\Lambda^s\widehat{g}\| \, \|\Lambda^{-s}\widehat{h}\|\,.
\end{align}
From \eqref{stimaSs}, \eqref{equSS*}, \eqref{stimaSigma*} (with $\Lambda^{-s-1}\widehat{h}$ instead of $\widehat{h}$) we obtain (notice that Lemma \ref{propsigma} and $\gamma\ge1$ give $|\sigma|>0$ )
\begin{equation}\label{stimaL-s}
\|\Lambda^{-s}\widehat{h}\|=\|\sigma\Lambda\sigma^{-1}\Lambda^{-s-1}\widehat{h}\|\le C\|\Sigma^\ast\sigma^{-1}\Lambda^{-s-1}\widehat h\| \le {C}\|\sigma^{-1}\Lambda^{-s+1}\widehat h\| \, .
\end{equation}
Let us denote
\[
\Rc:=\left\{ \Sigma^\ast\sigma^{-1}\Lambda^{-s-1}\widehat h \,\, | \,\, \sigma^{-1}\Lambda^{-s+1}\widehat h\in L^{2}(\R^2) \right\} \,.
 \]
 From \eqref{stimaL-s} it is clear that $ \Rc$ is a subspace of $ L^2(\R^2) $; moreover, the map $ \Sigma^\ast\sigma^{-1}\Lambda^{-s-1}\widehat h \mapsto \Lambda^{-s}\widehat{h} $ is well-defined and continuous from $ \Rc$ into $ L^2(\R^2) $. Given $ g\in H^s_\gamma(\R^2) $, we define a linear form $ \ell $ on $ \Rc $ by
 \[
 \ell(\Sigma^\ast\sigma^{-1}\Lambda^{-s-1}\widehat h)= \langle \widehat{g},\widehat{h} \rangle \,.
  \]
From \eqref{duality}, \eqref{stimaL-s} we obtain
\[
\left| \ell(\Sigma^\ast\sigma^{-1}\Lambda^{-s-1}\widehat h)\right| \le C\|\Lambda^s\widehat{g}\|  \, \|\Sigma^\ast\sigma^{-1}\Lambda^{-s-1}\widehat h\| \,.
 \]
 Thanks to the Hahn-Banach and Riesz theorems, there exists a unique $ w\in L^2(\R^2) $ such that
 \[
\langle w,\Sigma^\ast\sigma^{-1}\Lambda^{-s-1}\widehat h \rangle = \ell(\Sigma^\ast\sigma^{-1}\Lambda^{-s-1}\widehat h)\,, \qquad
 \|w\|= \|\ell\|_{\Lc(\Rc)} \le C\|\Lambda^s\widehat{g}\| \,.
  \]
Defining $ \widehat f:= \sigma^{-1}\Lambda^{-s-1}w$ we get $ \sigma\Lambda^{s+1}\widehat f= w\in L^2(\R^2)$, i.e. $ f\in H^{s+1,\sigma}_\gamma(\R^2) $ such that
\[
\langle \Sigma\widehat f, \widehat h \rangle = \langle \widehat f, \Sigma^\ast\widehat h \rangle= \langle \widehat{g},\widehat{h} \rangle \,,
 \]
for all $h$ such that $\sigma^{-1}\Lambda^{-s+1}\widehat h\in L^{2}(\R^2)$, which shows that $ f $ is a solution of equation \eqref{def_g}. Moreover
\[
\|f\|_{H^{s+1,\sigma}_\gamma(\R^2)}=\frac{1}{2\pi}\|\sigma\Lambda^{s+1}\widehat{f}\|=\frac{1}{2\pi}\|w\| \le C\|\Lambda^s\widehat{g}\|=C\|g\|_{H^{s}_\gamma(\R^2)}\,,
 \]
 that is \eqref{stimafg}. The uniqueness of the solution follows from the linearity of the problem and the a priori estimate.
\end{proof}

\begin{proof}[Proof of Theorem \ref{teoexist}]
Equation \eqref{equ_f} has the form of \eqref{def_g} with 
\[ \widehat g(\tau,\eta)=-\frac{\mu^+\mu^-}{\mu^++\mu^-}\,M \, , \]
for $M$ defined in \eqref{def_M}. We already know from \cite{MST18} that the assumptions on $\F^\pm$ yield $g\in H^s_\gamma(\R^2)$ and that $g$ satisfies the estimate
	\begin{equation}\label{stimagF}
	\begin{array}{ll}
	\ds \|g\|_{H^s_\gamma(\R^2)}^2
	\le \frac{C}{\gamma}\left( \|\F^+\|^2_{L^2(\R^+;H^s_\gamma(\R^2))}+\|\F^-\|^2_{L^2(\R^-;H^s_\gamma(\R^2))}\right) , \qquad\forall \gamma\ge1\, .
	\end{array}
	\end{equation}
Then, the conclusion of the proof follows from Theorem \ref{teoexistfg2} and \eqref{stimagF}.
\end{proof}

The following corollary of Theorem \ref{teoexist} shows that the existence result of \cite{MST18} is a simple consequence of the present result.
\begin{corollary}\label{corolla2}
Given the assumptions of Theorem \ref{teoexist}, there exists a unique solution $f\in H^{s+1}_\gamma(\R^2)$ of equation \eqref{equ_f}  satisfying the estimate
	\begin{equation}\label{stimafF1}
	\gamma^3 \|f\|^2_{H^{s+1}_\gamma(\R^2)} \le C\left( \|\F^+\|^2_{L^2(\R^+;H^s_\gamma(\R^2))}+\|\F^-\|^2_{L^2(\R^-;H^s_\gamma(\R^2))}\right) , \qquad\forall \gamma\ge1\, ,
	\end{equation}
	for a suitable constant $C>0$ independent of $\F^\pm$ and $\gamma$.
\end{corollary}
\begin{proof}
The result follows directly from Theorem \ref{teoexist} and the first inequality in \eqref{stimaspazi}.
\end{proof}

\subsection*{Acknowledgement}
The research was supported in part by the Italian research
project PRIN 2015YCJY3A-004.
The author would like to warmly thank Tao Wang for support and hospitality during
his visit at Wuhan University where this work was initiated.




\end{document}